\documentclass{amsart}

\usepackage{amsfonts}
\usepackage{amssymb}
\usepackage{amsmath}

\newcommand{\rd}{\mathrm{d}}
\newcommand{\rn}{\mathbb{R}^{N}}

\renewcommand{\O}{\Omega}
\newcommand{\PO}{\partial\Omega}

\newcommand{\lqb}{L^q(\partial\Omega)}
\renewcommand{\wp}{W^{1,p}(\Omega )}
\newcommand{\wpc}{W^{1,p}_0(\Omega )}
\newcommand{\wpa}{W^{1,p}_A(\Omega )}

\newtheorem{de}{Definition}[section]
\newtheorem{lem}[de]{Lemma}
\newtheorem{te}[de]{Theorem}
\newtheorem{co}[de]{Corollary}
\newtheorem{pr}[de]{Proposition}
\newtheorem{ob}[de]{Remark}

\def\cprime{$'$}

\providecommand{\MR}{\relax\ifhmode\unskip\space\fi MR }

\providecommand{\href}[2]{#2}

\begin{document}

\title[extremals of the trace inequality]
{Optimization problem for extremals of the trace inequality in domains with holes}
\author[Leandro M. Del Pezzo]
{Leandro M. Del Pezzo}

\address{Leandro M. Del Pezzo \hfill\break\indent
Departamento  de Matem\'atica, FCEyN, Universidad de Buenos Aires,
\hfill\break\indent Pabell\'on I, Ciudad Universitaria (1428),
Buenos Aires, Argentina.}

\email{{\tt ldpezzo@dm.uba.ar}}

\begin{abstract}
We study the Sobolev trace constant for functions defined in a bounded domain $\O$  that vanish in the subset $A.$  We find a formula for the first variation of the Sobolev trace with respect to hole. As a consequence of this formula, we prove that when $\O$ is a centered ball, the symmetric hole is critical when we consider deformation that preserve volume but is not optimal for some case. 
\end{abstract}

\maketitle

\section{Introduction and Main Results.}

Let $\O$ be a bounded smooth domain in $\rn$ with $N\geq2$ and
$1<p<\infty$. We denote by $p_*$ the critical exponent for the
Sobolev trace immersion given by $p_*=p(N-1)/(N-p)$ if $p<N$ and
$p_*=\infty$ if $p\ge N$.

For any $A\subset\overline{\O},$ which is a smooth open subset, we define the space
$$
\wpa=\overline{C_0^\infty(\overline{\O}\setminus A)},
$$
where the closure is taken in $W^{1,p}-$norm.
By the Sobolev Trace Theorem, there is a compact embedding
\begin{equation}\label{sobolev}
\wpa\hookrightarrow\lqb,
\end{equation}
for all $1<q<p^*$. Thus, given $1<q<p^*,$ there exist a constant $C=C(q,p)$ such that
$$
C\Bigg\{\int_{\PO} |u|^q \, \rd S\Bigg\}^{\frac{p}{q}} \leq
\int_\O |\nabla u|^p + |u|^p \, \rd x.
$$
The best (largest) constant in the above inequality is given by
\begin{equation}\label{sqa}
S_q(A):= \inf_{u\in\wpa\setminus\wpc}\dfrac{\int_\O |\nabla u|^p +
|u|^p \, \rd x}{\big\{\int_{\PO} |u|^q \, \rd
S\big\}^\frac{p}{q}}.
\end{equation}
By \eqref{sobolev}, there exist an extremal for $S_q(A).$
Moreover, an extremal for $S_q(A)$ is a weak solution to
\begin{equation}\label{roman}
\begin{cases}
-\Delta_p u + |u|^{p-2}u = 0 & \text{in } \O\setminus\overline{A}, \\
|\nabla u|^{p-2}\frac{\partial u}{\partial \nu} = \lambda |u|^{q-2}u & \text{on }\partial\O,\\
u=0 & \textrm{on }\partial A,
\end{cases}
\end{equation}
where $\Delta_p u = \text{div}(|\nabla u|^{p-2}\nabla u)$ is the
usual $p-$laplacian, $\frac{\partial}{\partial\nu}$ is the outer
unit normal derivative and $\lambda$ depends on the normalization
of $u$. When $\|u\|_{\lqb}=1$ we have that $\lambda=S_q(A).$
Moreover, when $p=q$ problem \eqref{roman} becomes homogeneous and
therefore is a nonlinear eigenvalue problem. In this case, the
first eigenvalue of \eqref{roman} coincides with the best Sobolev
trace constant $S_q(A)=\lambda_1(A)$ and it is shown in \cite{MR}
that it is simple (see also \cite{FBR2}). Therefore, if $p=q$, the
extremal for $S_p(A)$ is unique up to constant factor. In the
linear setting, i.e. when $p=q=2$, this eigenvalue problem is
known as the Steklov eigenvalue problem, see \cite{S}.

It is the purpose of this article to analyze the dependance of the
Sobolev trace constant $S_q(A)$ with respect to variations on the
set $A$. To this end, we compute the so-called {\em shape
derivative} of $S_q(A)$ with respect to regular perturbations of
the {\em hole} $A$.

Let $V:\rn\to\rn$ be a regular (smooth) vector filed, globally
Lipschitz, with support in $\O$ and let $\psi_t:\rn\to\rn$ be
defined as the unique solution to
\begin{equation}\label{palermo}
\begin{cases}
\frac{d}{dt}\psi_t(x)=V(\psi_t(x)) & t>0\\
\psi_0(x)=x & x\in\rn.
\end{cases}
\end{equation}
We have
$$
\psi_t(x) = x +t V(x) + o(t) \quad \forall x\in\rn.
$$
Now, we define $A_t:=\psi_t(A)\subset\O$ for all $t>0$ and
\begin{equation}\label{sqt}
S_q(t) = \inf_{u\in W^{1,p}_{A_t}(\Omega
)\setminus\wpc}\dfrac{\int_\O |\nabla u|^p + |u|^p \, \rd
x}{\big\{\int_{\PO} |u|^q \, \rd S\big\}^{\frac{p}{q}}}.
\end{equation}
Observe that $A_0 = A$ and therefore $S_q(0)=S_q(A)$.

In \cite{FBGR} Fern\'andez Bonder, Groisman and Rossi analyze this
problem in the linear case $p=q=2$ and prove that $S_2(t)$ is
differentiable with respect to $t$ at $t=0$ and it holds
$$
\frac{d}{dt}S_2(t)\Big|_{t=0} = -\int_{\partial A}
\left(\dfrac{\partial u}{\partial\nu}\right)^2\langle V, \nu
\rangle \, \rd S,
$$
where $u$ is a normalized eigenfunction for $S_2(A)$ and $\nu$ is
the exterior normal vector to $\O\setminus\overline{A}$.

Furthermore, in the case that $\O$ is the ball $B_R$ with center
$0$ and radius $R>0$ the authors show that a centered ball
$A=B_r$, $r<R$, is critical in the sense that $S_2'(A)=0$ when
considering deformations that preserves volume and that this
configuration is not optimal.

We say that hole $A^*$ is optimal for the parameter $\alpha$,
$0<\alpha<|\Omega|$, if $|A^*|=\alpha$ and
$$
S_q(A^*) = \inf_{A\subset\overline{\Omega} \atop |A|=\alpha}
S_q(A).
$$

Therefore there is a lack of symmetry in the optimal
configuration.

Here we extend these results to the more general case $1<p<\infty$
and $1<q<p^*$. Our method differs from the one in \cite{FBGR} in
order to deal with the nonlinear character of the problem.

Our first result states

\begin{te}\label{ibarra}
Suppose $A\subset\overline{\O}$ is a smooth open subset and let
$1<q<p^*$. Then, with the previous notation, we have that $S_q(t)$
is differentiable at $t=0$ and there exists $u$ a normalized
extremal for $S_q(A)$ such that
$$
S_q'(0) = -\int_{\partial A} \Big|\dfrac{\partial
u}{\partial\nu}\Big|^p\langle V, \nu \rangle \, \rd S,
$$
where $S_q'(0) = \frac{d}{dt}S_q(t)\Big|_{t=0}$ and $\nu$ is the
exterior normal vector to $\O\setminus\overline{A}.$
\end{te}

\begin{ob}\label{maradona}
If $u$ is an extremal for $S_q(A)$ we have that $|u|$ is also an
extremal associated to $S_q(A)$. Then in the previous theorem we
can suppose that $u\ge0$ in $\O$. Moreover, by \cite{L}, we have
that $u\in C^{1,\alpha}(\overline{\O})$ and if $\O$ satisfies the
interior ball condition for all $x\in\partial\O$ then $u>0$ on
$\partial \O$, see \cite{V}.
\end{ob}

In the case that $\O=B_R$, we have the next result

\begin{te}\label{datolo}
Let $\O=B_R$ and let the hole be a centered ball $A=B_r$. Then, if
$1<q\le p$, this configuration is critical in the sense that
$S_q'(B_r)=0$ for all deformations $V$ that preserve the volume of
$B_r$.
\end{te}

But, if $q$ is sufficiently large, the symmetric hole with a
radial extremal is not an optimal configuration. In fact, we prove

\begin{te}\label{morel}
Let $r>0$ and $1<p<\infty$ be fixed. Let $R>r$ and
\begin{equation}\label{Q}
Q(R) = \dfrac{1}{S_p(B_r)^{\frac{p}{p-1}}}
\left(1-\frac{N-1}{R}S_p(B_r)\right)+1.
\end{equation}
If $q>Q(R)$ then the centered hole $B_r$ is not optimal.
\end{te}

Finally, to study the asymptotic behavior of $Q(R)$

\begin{pr}\label{caranta}
The function $Q(R)$ has the following asymptotic behavior
$$
\lim_{R\to r}Q(R)=1^- \quad\textrm{and}\quad \lim_{R\to +\infty}Q(R)=p
$$
\end{pr}

Observe that $Q(R)<1$ for $R$ close to $r$ and therefore the
symmetric hole with a radial extremal is not an optimal
configuration for $R$ close to $r.$

\section{Proof  of Theorem \ref{ibarra}}

\subsection{Preliminary Results}

The proof of Theorem \ref{ibarra} require some technical results.
In this subsection we use some ideas from \cite{GMSL}.

Given $u\in W^{1,p}_{A_t}(\Omega )\setminus\wpc$  we consider $v =
u \circ\psi_t$, so $v\in\wpa\setminus\wpc$ and $\nabla v^T= \,
^T\psi_t'\nabla(u \circ\psi_t)^T,$ where $\psi_t'$ denotes the
differential matrix of $\psi_t$ and $^T A$ is the transpose of
matrix A. Thus, by the change of variables formula, we have that
$$
\int_\O|\nabla u|^p + |u|^p \, \rd x = \int_\O \{|^T[\psi_t']^{-1}\nabla v^T|^p + |v|^p\}J(\psi_t)\,\rd x,
$$
here $J(\psi_t)$ is the usual Jacobian of $\psi_t$. Moreover, since supp$(V)\subset \O$, we have that
$$
\int_{\PO}|u|^q\,\rd S = \int_{\PO} |v|^q \, \rd S.
$$

In \cite{Henrot} are proved the following asymptotic formulas
\begin{align}
\label{inversa} [\psi_t']^{-1}(x) &= Id - tV'(x) + o(t),\\
\label{jacobiano} J(\psi_t)(x) &= 1 + t\,\textrm{div} V(x) + o(t).
\end{align}

Then, by \eqref{inversa} and \eqref{jacobiano}, we have
\begin{eqnarray*}
\int_\O |v|^pJ(\psi_t)\,\rd x &=& \int_\O |v|^p\{1 + t\,\textrm{div} V + o(t)\}\,\rd x\\
&=&\int_\O |v|^p\,\rd x +t\int_\O |v|^p\textrm{div} V\,\rd x + o(t)
\end{eqnarray*}
and
\begin{align*}
\int_\O |^T[\psi_t']^{-1}\nabla v^T|^p J(\psi_t)\,\rd x &
= \int_\O |[Id - t \,^T V' + o(t)]\nabla v^T|^p \{1 + t\,\textrm{div} V + o(t)\} \, \rd x\\
&=\int_\O |\nabla v - t \, ^T V'\nabla v^T  + o(t)|^p\{1 + t\,\textrm{div} V + o(t)\} \, \rd x,
\end{align*}
since
$$
|\nabla v - t  \,^T V'\nabla v^T   + o(t)|^p=|\nabla v|^p
-pt|\nabla v_t|^{p-2}\langle\nabla v,\, ^T V'\nabla v^T \rangle+
o(t)
$$
we obtain that
\begin{align*}
\int_\O |^T[\psi_t']^{-1}\nabla v^T|^p J(\psi_t)\,\rd x&=\int_\O
|\nabla v|^p\,\rd x + t\int_\O|\nabla v|^p\textrm{div}V\,\rd x\\
&-pt\int_\O|\nabla v|^{p-2}\langle\nabla v,\,  ^T V' \nabla v^T
\rangle\,\rd x + o(t).
\end{align*}
Thus, we conclude
\begin{align*}
\int_\O|\nabla u|^p + |u|^p \, \rd x &= \int_\O \{|^T[\psi_t']^{-1}\nabla v^T|^p + |v|^p\}J(\psi_t)\,\rd x\\
&=\int_\O |v|^p\,\rd x  + \int_\O |\nabla v|^p\,\rd x
+t\int_\O\{|\nabla v|^p+|v|^p\}\textrm{div} V \, \rd x \\
&-pt\int_\O|\nabla v|^{p-2}\langle\nabla v,\, ^TV'\nabla v^T\rangle\,\rd x + o(t).
\end{align*}
Therefore, we can rewrite \eqref{sqt} as
\begin{equation}\label{sq}
S_q(t)=\inf_{v\in\wpa\setminus\wpc}\{\rho(v) + t\gamma(v)\}
\end{equation}
where
$$
\rho(v) = \dfrac{ \int_\O |\nabla v|^p + |v|^p \, \rd
x}{\Big\{\int_{\PO} |v|^q \, \rd S\Big\}^{p/q}} ,
$$
and
$$
\gamma(v) = \dfrac{\int_\O \{|\nabla v|^p + |v|^p\}\textrm{div} V
\, \rd x -p\int_\O |\nabla v|^{p-2}\langle\nabla v,\, ^TV'\nabla
v^T\rangle \, \rd x}{\Big\{\int_{\PO} |v|^q\,\rd S\Big\}^{p/q}} +
O(t).
$$

Given $t\ge0$, let $v_t\in\wpa\setminus\wpc$ such that $\|v_t\|_{\lqb}=1$ and
$$
S_q(t)=\varphi(t) + t\phi(t),
$$
where
$$
\varphi(t) = \rho(v_t) \textrm{ and } \phi(t) = \gamma(v_t)
\quad\forall t\ge0.
$$
We observe that $\varphi,\phi:\mathbb{R}_{\ge0}\to\mathbb{R}$ and

\begin{lem}\label{decreciente}
The function $\phi$ is nonincreasing.
\end{lem}

\begin{proof}
Let $0\le t_1\le t_2$.  By \eqref{sq}, we have that
\begin{eqnarray}
\label{1d}\varphi(t_2) + t_1 \phi (t_2) &\ge& S_q(t_1)=\varphi (t_1) + t_1 \phi(t_1)\\
\label{2d}\varphi(t_1) + t_2 \phi (t_1) &\ge& S_q(t_2)=\varphi (t_2) + t_2 \phi(t_2).
\end{eqnarray}
Subtracting \eqref{1d} from \eqref{2d}, we get
$$
(t_2 - t_1 ) \phi(t_1) \ge (t_2 - t_1 ) \phi(t_2).
$$
Since $t_2 - t_1 \ge 0,$ we obtain
$$
 \phi(t_1) \ge \phi(t_2).
$$
This ends the proof.
\end{proof}

\begin{ob}\label{acotado}
Since $\phi$ is nonincreasing, we have
$$
\phi(t)\le\phi(0) \quad \forall t\ge0,
$$
and there exists
$$
\phi(0^+)=\lim_{t\to0^+}\phi(t).
$$
\end{ob}

\begin{co}\label{creciente}
The function $\varphi$ is nondecreasing.
\end{co}

\begin{proof}  Let $0\le t_1\le t_2.$ Again, by \eqref{sq}, we have that
\begin{equation}
\label{3d}\varphi(t_2) + t_1 \phi (t_2) \ge S_q(t_1)=\varphi (t_1) + t_1 \phi(t_1)
\end{equation}
so
$$
\varphi(t_2)-\varphi(t_1)\ge t_1(\phi(t_1)-\phi(t_2)).
$$
Since $0\le t_1\le t_2,$ by Lemma \ref{decreciente}, we have that $\phi(t_1)-\phi(t_2)\ge0.$ Then
$$
\varphi(t_2)-\varphi(t_1)\ge0
$$
that is what we wished to prove.
\end{proof}

Now we can prove that $S_q(t)$ is continuous at $t=0.$

\begin{te}\label{continuidad}
The function $S_q(t)$ is continuous at $t=0,$ i.e.,
$$
\lim_{t\to0^+}S_q(t)=S_q(0).
$$
\end{te}

\begin{proof}
Given $t\ge0$ so, by Corollary \ref{creciente},
$$
S_q(t)-Sq(0)= \varphi(t)+t\phi(t)-\varphi(0)\ge t\phi(t).
$$
On the other hand, by \eqref{sq}, we have that
$$
S_q(t)\le \varphi(0)+t\phi(0)=S_q(0) + t\phi(0).
$$
Then
$$
t\phi(t)\le S_q(t)-Sq(0) \le  t\phi(0).
$$
Thus, by Remark \ref{acotado},
$$
\lim_{t\to 0^+} S_q(t) - S_q(0)=0.
$$
This finishes  the proof.
\end{proof}

Thus, from Remark \ref{acotado} and Theorem \ref{continuidad}, we obtain the following corollary:

\begin{co}
The function $\varphi$ is continuous at $t=0,$ i.e.,
$$
\lim_{t\to0^+}\varphi(t)=\varphi(0).
$$
\end{co}

\begin{proof}\label{cocont}
We observe that
$$
\varphi(t)-\varphi(0)= S_q(t)-S_q(0) - t \phi(t)
$$
then, by Remark \ref{acotado} and Theorem \ref{continuidad},
$$
\lim_{t\to0^+}\varphi(t)-\varphi(0)=0.
$$
That proves the result.
\end{proof}

Finally, we prove the following:

\begin{te}\label{diferenciable}
The function $\varphi$ is differentiable at $t=0$ and
$$
\frac{\rd\varphi}{\rd t}(0)=0.
$$
\end{te}

\begin{proof}
Let $0<r<t$. By \eqref{sq}, we get
$$
S_q(r) = \varphi(r) + r \phi(r)\le \varphi(t) + r \phi(t),
$$
and
$$
S_q(t) = \varphi(t) + t \phi(t)\le \varphi(r) + t \phi(r).
$$
So
$$
\dfrac{r}{t}(\phi(r)-\phi(t)) \le
\dfrac{\varphi(t)-\varphi(r)}{t}\le \phi(r)-\phi(t)
$$
hence, taking limits when $r\to0^+$, by Remark \ref{acotado} and Corollary \ref{cocont}, we have that
$$
0 \le \dfrac{\varphi(t)-\varphi(0)}{t}\le \phi(0^+)-\phi(t).
$$
Now, taking limits when $t\to0^+$, and again by Remark \ref{acotado}, we get
$$
\lim_{t\to0^+}\dfrac{\varphi(t)-\varphi(0)}{t}=0
$$
as we wanted to show.
\end{proof}

\subsection{Proof of Theorem \ref{ibarra}} We proceed in three steps.

\noindent\textbf{Step 1}. We show that $S_q(t)$ is differentiable
at $t=0$ and
$$
S_q'(0) = \phi(0^+).
$$
We have that
$$
\dfrac{S_q(t)-S_q(0)}{t}=\dfrac{\varphi(t)-\varphi(0)}{t} - \phi(t).
$$
Then, by Remark \ref{acotado} and Theorem \ref{diferenciable},
$$
S_q'(0) =\lim_{t\to0^+}\dfrac{S_q(t)-S_q(0)}{t}=\phi(0^+).
$$

\noindent\textbf{Step 2.} We show that there exists $u$ extremal
for $S_q(A)$ such that $\|u\|_{\lqb} =1$ and
$$
\phi(0^+) = \int_\O (|\nabla u|^p +|u|^p)\textnormal{div}V \, \rd
x- p\int_\O |\nabla u|^{p-2}\langle\nabla u, \,^TV' \nabla u
\rangle \, \rd x.
$$
By Theorem \ref{cocont}
\begin{equation}
\label{con}\|v_t\|^p_{\wp} = \varphi(t)\to\varphi(0)=S_q(0)
\textrm{ when }t\to0^+.
\end{equation}
Then there exists $u\in\wp$ and $t_n\to 0^+$ when $n\to\infty$ such that
\begin{eqnarray}
\label{code}v_{t_n}&\rightharpoonup& u \textrm{ weakly in } \wp,\\
\label{cofu}v_{t_n}&\to& u \textrm{ strongly in } \lqb,\\
\label{copu}v_{t_n}&\to& u \textrm{ a.e. in  } \O.
\end{eqnarray}
By \eqref{cofu} and \eqref{copu}, $u\in\wpa$ and $\|u\|_{\lqb}=1$ and by \eqref{code}
$$
S_q(0)=\lim_{n\to\infty}\|v_{t_{n}}\|^p_{\wp}\ge\|u\|^p_{\wp}\ge S_q(0),
$$
then
\begin{equation}
\label{igual}S_q(0)=\|u\|^p_{\wp}.
\end{equation}
Moreover, by \eqref{con}, \eqref{code} and \eqref{igual}, we have that
$$
v_{t_n}\to u \textrm{ strongly in } \wp.
$$
Therefore
\begin{align*}
\phi(0^+)&=\lim_{n\to\infty}\phi(v_{t_n})\\
& = \int_\O (|\nabla u|^p +|u|^p)\textnormal{div}V \, \rd x-
p\int_\O |\nabla u|^{p-2}\langle\nabla u,  \,^TV'\nabla u^T
\rangle \, \rd x.
\end{align*}

\noindent\textbf{Step 3}. Finally, we show that
\begin{align*}
S_q'(0) =&\int_\O (|\nabla u|^p +|u|^p)\textnormal{div}V \, \rd x -
p\int_\O |\nabla u|^{p-2}\langle\nabla u,  \,^TV' \nabla u^T \rangle \, \rd x\\
=& -\int_{\partial A} \Big|\frac{\partial u}{\partial\nu}\Big|^p\langle V, \nu \rangle \, \rd S.
\end{align*}

To show this we require that $u\in C^2.$ However, this is not true. Since $u$ is an esxtremal for $S_q(A)$ and $\|u\|_{L^q(\O)}=1$, we known that $u$ is weak solution to 
$$
\begin{cases}
-\Delta_p u + |u|^{p-2}u = 0 & \text{in } \O\setminus\overline{A}, \\
|\nabla u|^{p-2}\frac{\partial u}{\partial \nu} = S_q(A) |u|^{q-2}u & \text{on }\partial\O,\\
u=0 & \textrm{on }\partial A,
\end{cases}
$$
and by \cite{L} we get that $u$ belongs to the class $C^{1,\delta}$ for some $0<\delta<1.$

In order to overcome this difficulty, we proced as follows. We consider the regularized prblems
\begin{equation}\label{rusas}
\begin{cases}
-\textrm{div}(|\nabla u^{\varepsilon}|^2+\varepsilon^2)^{(p-2)/2}) + |u^\varepsilon|^{p-2}u^\varepsilon = 0 & \text{in } \O\setminus\overline{A}, \\
|\nabla u^{\varepsilon}|^2+\varepsilon^2)^{(p-2)/2}\frac{\partial u^\varepsilon}{\partial \nu} = S_q(A) |u|^{q-2}u & \text{on }\partial(\O\setminus\overline{A}),
\end{cases}
\end{equation}
It is well known that the solution $u^\varepsilon$ to \eqref{rusas} is of class $C^{2,\rho}$ for some $0<\rho<1$ (see \cite{LSU}).

Then, we can perform all of our computations with the functions $u^\varepsilon$ and pass to the limit as $\varepsilon\to0$ at the end.

We have chosen to work formally with the function $u$ in order to make our arguments more transparent and leave the details to the reader. For a similar approach, see \cite{GMSL}.

Since
\begin{eqnarray*}
\textnormal{div}(|u|^pV)&=&|u|^p\textnormal{div}V + p|u|^{p-2}u\langle\nabla u, V\rangle,\\
\textnormal{div}(|\nabla u|^pV)&=&|\nabla u|^p\textnormal{div}V + p|\nabla u|^{p-2}\langle\nabla u D^2 u, V\rangle,
\end{eqnarray*}
we have that
\begin{align*}
\int_\O (|\nabla u|^p +|u|^p)\textnormal{div}V \, \rd x =& \int_\O \textnormal{div}(|u|^pV+|\nabla u|^pV)\, \rd x\\
-&p\int_\O\{|u|^{p-2}u_0\langle\nabla u, V\rangle+|\nabla u|^{p-2}\langle\nabla u D^2 u, u V\rangle \, \}\rd x.
\end{align*}
Integrating by parts, we obtain
\begin{align*}
\int_\O \textnormal{div}(|u|^pV+|\nabla u|^pV)\, \rd x=&\int_{\PO}(|u|^p+|\nabla u|^p)\langle V,\nu\rangle \, \rd S -\int_{\partial A}(|u|^p+|\nabla u|^p)\langle V,\nu\rangle \, \rd S\\
=&- \int_{\partial A}|\nabla u|^p\langle V,\nu\rangle \, \rd S.
\end{align*}
where the las equality follows from the fact that supp$(V)\subset\O$ and $u=0$ on $\partial A$.

Thus
\begin{align*}
S_q'(0)
=&- \int_{\partial A}|\nabla u|^p\langle V,\nu\rangle \, \rd S
-p\int_\O|u|^{p-2}u\langle\nabla u_0, V\rangle \rd x\\
-&p\int_\O|\nabla u|^{p-2}\langle \nabla u,\,^TV'\nabla u + ^TD^2 u V^T\rangle \, \rd x\\
=&- \int_{\partial A}|\nabla u|^p\langle V,\nu\rangle \, \rd S
-p\int_\O|u_0|^{p-2}u\langle\nabla u, V\rangle \rd x\\
-&p\int_\O|\nabla u|^{p-2}\langle \nabla u,\,\nabla(\langle\nabla u,V\rangle)\rangle \, \rd x.
\end{align*}
Since $u$ is a week solution of \eqref{roman} as $\lambda = S_q(0)$ and supp$(V)\subset\O$ we have
$$
S_q'(0)=- \int_{\partial A}|\nabla u|^p\langle V,\nu\rangle \, \rd S.
$$
Then, noticing that $\nabla u = \frac{\partial u}{\partial\nu}\nu,$ the proof is complete.\hfill$\square$

\section{Lack of Symmetry in the Ball}

In this section we consider the case where $\O =B_R$ and $A=B_r$
with $r<R$ and show Theorem \ref{datolo}, Theorem \ref{morel} and
Proposition \ref{caranta}. The proofs are based on the argument of
\cite{FBGR} and \cite{LDT} adapted to our problem. In order to
simplify notations, we write $S_q(r)$ instead $S_q(B_r).$

First we proof Theorem \ref{datolo}, for this we need the following proposition

\begin{pr}\label{unicidad}
Let $1<q<p$. The nonnegative solution of \eqref{roman} is unique.
\end{pr}
\begin{proof}
Suppose that there exist two nonnegative solutions $u$ and $v$ of
\eqref{roman}. By Remark \ref{maradona} it follows that $u,$ $v>0$
on $\partial\O.$ Let $v_n=v+\frac{1}{n}$ with $n\in\mathbb{N}$,
using first Piccone's identity (see \cite{AH}) and the weak
formulation of \eqref{roman} we have
\begin{align*}
0 & \le \int_{B_R}|\nabla u|^p\,\rd x -
\int_{B_R}|\nabla v_n|^{p-2}\nabla v_n \nabla\left(\dfrac{u^p}{v_n^{p-1}}\right)  \,\rd x\\
& = \int_{B_R}|\nabla u|^p\,\rd x -
\int_{B_R}|\nabla v|^{p-2}\nabla v \nabla\left(\dfrac{u^p}{v_n^{p-1}}\right)\,\rd x\\
& = -\int_{B_R}u^p \, \rd x + \lambda\int_{\partial B_R} u^q \, \rd S +
\int_{B_R}v^{p-1}\dfrac{u^p}{v_n^{p-1}}\,\rd x -
\lambda\int_{\partial B_R} v^{q-1}\dfrac{u^{p-1}}{v_n^{p-1}}\, \rd S\\
& \le \lambda\int_{\partial B_R} u^q \, \rd S -
\lambda\int_{\partial B_R} v^{q-1}\dfrac{u^{p-1}}{v_n^{p-1}}\, \rd
S.
\end{align*}

Thus, by the Monotone Convergence Theorem,
\begin{align*}
0&\le \int_{\partial B_R} u^q \, \rd S -\int_{\partial B_R} v^{q-1}\dfrac{u^{p-1}}{v^{p-1}}\, \rd S\\
&=\int_{\partial B_R} u^q (u^{q-1} -v^{q-1})\, \rd S.
\end{align*}

Note that the role of $u$ and $v$ in the above equation are
exchangeable. Therefore, subtracting we get
$$
0\le \int_{\partial B_R} (u^q-v^q) (u^{q-1} -v^{q-1})\, \rd S.
$$
Since $q<p$ we have that $u\equiv v$ on $\partial B_R$. Then, by
uniqueness of solution to the Dirichlet problem, we get $u\equiv
v$ in $B_R$.
\end{proof}

\begin{ob}\label{radial}
As the problem \eqref{roman} is rotationally invariant, by
uniqueness we obtain that the nonnegative solution of
\eqref{roman} must be radial. Therefore, if $\O=B_R,$ $A=B_r$ and
$1<q\le p$ we can suppose that the extremal for $S_q(r)$ found in
the Theorem \ref{ibarra} is nonnegative and radial.
\end{ob}

Now we can prove the Theorem \ref{datolo},

\begin{proof}[\textbf{Proof of Theorem \ref{datolo}}]
We consider $\O=B_R,$ $A=B_r$ and $1<q\le p.$ By Theorem
\ref{roman} and Remark \ref{radial} there exist a nonnegative and
radial normalized extremal for $S_q(r)$ such that
$$
S_q'(0) = -\int_{\partial B_r} \Big|\frac{\partial
u}{\partial\nu}\Big|^p\langle V, \nu \rangle \, \rd S.
$$
Since $u$ is radial
$$
\frac{\partial u}{\partial\nu}\equiv c \textrm{ on } \partial B_r,
$$
where $c$ is a constant.

Thus, using that we are dealing with deformations $V$ that
preserves the volume of the $B_r$, we have that
$$
S_q'(0) = -c^p\int_{\partial B_r} \langle V, \nu \rangle \, \rd S
= c^p\int_{B_r}div(V)\,\rd x=0.
$$
\end{proof}

To prove Theorem \ref{morel}, we need two previous results.

\begin{pr}\label{u0}
Let $r>0$ fixed. Then, there exists a positive radial function
$u_0$ such that
\begin{equation}\label{a1}
\begin{cases}
-\Delta_p u + |u|^{p-2}u = 0 & \text{in } \rn\setminus B_r, \\
u=0 & \textrm{on }\partial B_r.
\end{cases}
\end{equation}
This $u_0$ is unique up to a constant factor and for any $R>r$ the
restriction of $u_0$ to $B_R$ is the first eigenfunction of
\eqref{roman} with $q=p.$
\end{pr}

\begin{proof}
For $R>r$, let $u_R$ be the unique solution of the Dirichlet
problem
$$
\begin{cases}
\Delta_p u_R = |u_R|^{p-2}u_R& \text{in } B_R\setminus\overline{B_r}, \\
u(R) = 1, & \\
u(r)=0.
\end{cases}
$$

Then, by uniqueness, $u_R$ is a nonnegative and radial function.
Moreover, by the regularity theory and maximum principle we have
$\frac{\partial u_R}{\partial \nu}(r)\neq 0$ (see \cite{L,V}).
Thus, for any $R>r$, we define the restriction of $u_0$ by
$$
u_0= \dfrac{u_R}{\dfrac{\partial u_R}{\partial \nu}(r)}.
$$
By uniqueness of the Dirichlet problem, it is easy to check that
$u_0$ is well defined and is a nonnegative radial solution of
\eqref{a1}. Furthermore, by the simplicity of $S_p(r)$, $u_0$ is
the eigenfunction associated to $S_p(r)$ for every $R>r$.
\end{proof}

\begin{pr}\label{gago}
Let $v$ be a radial solution of \eqref{roman}. Then $v$ is a
multiple of $u_0$. In particular any radial minimizer of
\eqref{sqa} is a multiple of $u_0.$
\end{pr}

\begin{proof}\label{importante}
Let $a>0$ be such that $v=au_0$ on $\partial B(0,R)$. Then $v$ and
$au_0$ are two solutions to the Dirichlet problem $\Delta_p w =
w^{p-1}$ and $w=v$ on
$\partial\left(B_R\setminus\overline{B_r}\right)$. Hence, by
uniqueness, we have that $v=au_0$ in $B_R.$
\end{proof}

\begin{ob}
If $1<q<p$ then the solution of \eqref{roman}, by Remark
\ref{radial} and Proposition \ref{importante}, is a multiple of
$u_0.$
\end{ob}

Now we can deal with the proof of Theorem \ref{morel}.

\begin{proof}[\textbf{Proof of Theorem \ref{morel}}]
Let $R>r$ be fixed and consider $u_0$ to be the nonnegative radial
function given by Proposition \ref{u0} such that that $u_0 =1$ on
$\partial B_R$. Then, by Proposition \ref{gago}, it is enough to
prove that $u_0$  is not a minimizer for $S_q(r)$ when $q>Q(R).$

First let us move this symmetric configuration in the $x_1$
direction. For any $t\in\mathbb{R}$ and $x\in\rn$ we denote
$x_t=(x_1-t,x_2,\dots,x_N)$ and define
$$
U(t)(x)=u_0(x_t)
$$
Observe that $U$ vanishes in $A_t:=B_r(te_1)$ (the ball with
center $te_1$ and radius $r$) a subset of $B_R$ of the same
measure of $B_r$ for all $t$ small.

Consider the function
$$
h(t)=\dfrac{f(t)}{g(t)}
$$
where
$$
f(t)=\int_{B_R}|\nabla U|^p + U^p \, \rd x \quad\textrm{and}\quad
g(t)=\left(\int_{\partial B_R} U^q \rd S\right)^{\frac{p}{q}}.
$$
We observe that $h(0)=0$ and since $h$ is an even function, we
have $h'(0)=0.$ Now,
$$
h''(0)=\dfrac{f''g^2-fgg''-2f'gg'-2fgg'}{g^3}\bigg|_{t=0}.
$$
Next we compute these terms. First, since $u_0$ is the first
eigenfunction of \eqref{roman} with $q=p$ and $u_0=1$ on $\partial
B_R$ we get
$$
f(0)=S_p(r)|\partial B_R| \quad \textrm{and} \quad g(0)=|\partial B_R|^{\frac{p}{q}}.
$$
Thus, by Gauss--Green's Theorem and using the fact that $u_0$ is
radial, we get
$$
f'(0) = -\int_{B_R}\dfrac{\partial }{\partial x_1}\left(|\nabla
u_0|^p + u_0^p\right) \rd x = \int_{\partial B_R}(|\nabla u_0|^p +
u_0^p)\nu_1\rd S= 0.
$$
Again, since $u_0$ is radial,
$$
g'(0) = \dfrac{p}{q}\left(\int_{\partial B_R}u^q \rd S
\right)^{\frac{p}{q}-1}\left(\int_{\partial B_R}\dfrac{\partial
u^q}{\partial x_1}\rd S\right)=0.
$$
Finally, using that $u_0=1$ on $\partial B_R$, we obtain
$$
g''(0) = p|\partial B_R|^{\frac{p}{q}-1}\int_{\partial
B_R}(q-1)\left(\dfrac{\partial u_0}{\partial
x_1}\right)^2+\dfrac{\partial^2 u_0}{\partial x_1^2} \,\rd S
$$
and, by the Gauss--Green's Theorem
\begin{align*}
f''(0) &= p\int_{B_R}\dfrac{\partial}{\partial x_1}
\left(\dfrac{1}{2}|\nabla u_0|^{p-2}\dfrac{\partial |\nabla u_0|^2}{\partial x_1} +
\dfrac{1}{p}\dfrac{\partial u_0^p}{\partial x_1}\right)\,\rd x\\
& =p \int_{\partial B_R}\left(\dfrac{1}{2}|\nabla
u_0|^{p-2}\dfrac{\partial |\nabla u_0|^2}{\partial x_1} +
\dfrac{1}{p}\dfrac{\partial u_0^p}{\partial x_1}\right)\nu_1 \,
\rd S.
\end{align*}
Then
\begin{align*}
h''(0) &= \dfrac{p}{|\partial B_R(0)|^{p/q}}
\Bigg[ \int_{\partial B_R}\left(\dfrac{1}{2}|\nabla u_0|^{p-2}\dfrac{\partial |\nabla u_0|^2}{\partial x_1} +
\dfrac{1}{p}\dfrac{\partial u_0^p}{\partial x_1}\right)\nu_1 \, \rd S\\
& \quad - S_p(r) \int_{\partial B_R}(q-1)\left(\dfrac{\partial
u_0}{\partial x_1}\right)^2+\dfrac{\partial^2 u_0}{\partial x_1^2}
\,\rd S \Bigg].
\end{align*}
Thus, since $u_0$ is radial, we get
\begin{align*}
h''(0) &= \dfrac{p}{N|\partial B_R(0)|^{p/q}}\Bigg[
\int_{\partial B_R}\left(\dfrac{1}{2}|\nabla u_0|^{p-2}\dfrac{\partial |\nabla u_0|^2}{\partial \nu} +
\dfrac{1}{p}\dfrac{\partial u_0^p}{\partial \nu}\right) \, \rd S\\
& \quad- S_p(r)\int_{\partial B_R}(q-1)|\nabla u_0|^2+ \Delta u_0
\,\rd S \Bigg].
\end{align*}
Now, by definition, $u_0(x)=u_0(|x|)$ and $\alpha$ satisfies
$$
(s^{N-1}|u_0'|^{p-1}u_0')' = s^{N-1}u_0^{p-1} \quad \forall s>r
$$
with $u_0(R)=0$ and $u_0(r)=0,$ moreover, by Proposition \ref{u0}, we have
$$
u_0'(s)^{p-1}=S_p(r)u_0(s)^{p-1} \quad \forall s>r.
$$
Then
$$
\dfrac{1}{2}|\nabla u_0|^{p-2}\dfrac{\partial |\nabla
u_0|^2}{\partial \nu} + \dfrac{1}{p}\dfrac{\partial
u_0^p}{\partial \nu}= \dfrac{S_p(r)^{\frac{1}{p-1}}}{p-1}
\left(1-\dfrac{N-1}{R}S_p(r)\right) + S_p(r)^{\frac{1}{p-1}}
$$
and
\begin{align*}
S_p(r)\left[ (q-1)|\nabla u_0|^2+ \Delta u_0\right] &=
(q-1)S_p(r)^{\frac{p+1}{p-1}} + \dfrac{S_p(r)^\frac{1}{p-1}}{p-1}
\left(1-\dfrac{N-1}{R}S_p(r)\right)\\
& + \dfrac{N-1}{R}S_p(r)^{\frac{p}{p-1}}.
\end{align*}
Therefore
$$
h''(0) = \dfrac{p S_p^{\frac{1}{p-1}}}{N |\partial
B_R|^{\frac{p}{q}-1}}\left[1-(q-1)S_p(r)^{\frac{p}{p-1}}-\dfrac{N-1}{R}S_p(r)
\right].
$$

Thus, if $q>Q(R)$ we get that $h''(0)<0$ and so 0 is a strict
local maxima of $\psi$. So we have proved that
$$
S_q(r)=h(0)>h(t)\ge S_q(B_r(te_1))
$$
for all t small. Therefore a symmetric configuration is not optimal.
\end{proof}

To finish the paper we prove Proposition \ref{caranta}.

\begin{proof}[\textbf{Proof of Proposition \ref{caranta}}]
We proceed in two step.

\noindent\textbf{Step 1.} First we show that, for $R>r,$ $S_p(R,r)=S_p(r)$ verifies the differential equation
\begin{equation}\label{deriv}
\dfrac{\partial S_p}{\partial R} = -\dfrac{N-1}{R}S_p + 1 -
(p-1)S_p^{\frac{p}{p-1}}
\end{equation}
with the condition
$$
S_p|_{R=r}=+\infty.
$$

Again we consider $u_0(x)=u_0(|x|)$ the nonnegative radial
function given by Proposition \ref{u0}. Thus, for all $R>r,$ we
get
$$
\begin{cases}
(p-1)\left(u_0 '\right)^{p-2}u_0''+\dfrac{N-1}{R}(u_0')^{p-1}= u_0^{p-1}, &\\
u_0'(R)^{p-1}=S_pu_0(R)^{p-1}, & \\
u_0(r)=0.
\end{cases}
$$
Then
$$
S_p=\left( \dfrac{u_0'(R)}{u_0(R)}\right)^{p-1}.
$$
Thus
\begin{align*}
\frac{\partial S_p}{\partial R} &= (p-1)\left( \dfrac{u_0'(R)}{u_0(R)}\right)^{p-2}
\dfrac{u_0''(R)u_0(R)-u_0'(R)^2}{u_0(R)^2}\\
&=(p-1)\left( \dfrac{u_0'(R)}{u_0(R)}\right)^{p-2}\dfrac{u_0''(R)}{u_0(R)}-(p-1)S_p^{\frac{p}{p-1}}\\
&=(p-1)\dfrac{u_0'(R)^{p-2}u_0''(R)}{u_0(R)^{p-1}}-(p-1)S_p^{\frac{p}{p-1}}\\
&=1-\frac{N-1}{R}S_p-(p-1)S_p^{\frac{p}{p-1}}.
\end{align*}

On the other hand, since (by definition) $\frac{\partial
u_0}{\partial\nu}\equiv1$ on $\partial B_r$, we get that
$u'(r)=1$. Then
$$
\lim_{R\to r}S_p = \lim_{R\to r}\left(
\dfrac{u_0'(R)}{u_0(R)}\right)^{p-1} = +\infty.
$$

Now, it is easy to check that $\lim_{R\to r}Q(R)=1^-$.

\noindent\textbf{Step 2.} Finally, we prove that
$$
\lim_{R\to+\infty}Q(R)=p.
$$

We begin  differentiating \eqref{deriv} to obtain
$$
\dfrac{\partial^2 S_p}{\partial R^2} = \dfrac{N-1}{R^2}S_p
-\dfrac{N-1}{R}\dfrac{\partial S_p}{\partial
R}-pS_p^{\frac{1}{p-1}}\dfrac{\partial S_p}{\partial R}.
$$
Then, since $S_p>0$, at any critical point ($S_p'=0$) we have that
$S_p''>0.$ Thus, $S_p$ has at most one critical point, which is a
minimum. If $S_p$ has a minimum, then there exist $R_0>r$ such
that $S_p'(R_0)=0$. Moreover, since  $S_p'(R)\neq0$ for any $R\neq
R_0$ and $S_p\to+\infty$ as $R\to r$ and by \eqref{deriv}, we get
that $S_p'<0$ for all $r<R<R_0$ and $S_p'>0$ for all $R>R_0$.
Thus, using again \eqref{deriv} we have that
$S_p^{\frac{p}{p-1}}<\frac{1}{p-1}$ for all $R>R_0.$ Then $S_p$ is
strictly increasing as a function of $R$ and bonded for all
$R>R_0.$ Consequently $S_p'\to 0$ as $R\to+\infty.$ It follows, by
\eqref{deriv}, that $S_p^{\frac{p}{p-1}}\to\frac{1}{p-1}$ as
$R\to+\infty.$ On the other hand using \eqref{Q} and \eqref{deriv}
we see that
\begin{equation}
S_p=(Q(R)-p)S_p^{\frac{p}{p-1}}.
\end{equation}
So, if $S_p$ has a minimum, we get that $Q(R)>p$ for all $R>R_0$
and $Q(R)\to p^{+}$ as $R\to+\infty.$ Now, If $S_p$ has not
critical points so $S_p'\neq0$  for all $R>r$ and using that
$S_p\to+\infty$ as $R\to r$ and \eqref{deriv} we get that $S_p'<0$
for all $R>r.$ Consequently, in this case, $S_p$ is strictly
decreasing and therefore $S_p'\to 0$ as $R\to +\infty$ and by
\eqref{deriv} we have that $S_p\to\frac{1}{p-1}$ as $R\to+\infty$.
Then, if $S_p$  has not critical points, we get $Q(R)<p$ and
$Q(R)\to p^{-}$ as $R\to+\infty$.
\end{proof}

\textbf{Acknowledgements}
I want to thank J. Fern\'andez Bender
for his throughout reading of the manuscript that help us to
improve the presentation of paper.

\bibliographystyle{amsplain}

\begin{thebibliography}{10}

\bibitem{AH}
Walter Allegretto and Yin~Xi Huang, \emph{A {P}icone's identity for the
  {$p$}-{L}aplacian and applications}, Nonlinear Anal. \textbf{32} (1998),
  no.~7, 819--830.

\bibitem{FBGR}
Juli{\'a}n~Fern{\'a}ndez Bonder, Pablo Groisman, and Julio~D. Rossi,
  \emph{Optimization of the first {S}teklov eigenvalue in domains with holes: a
  shape derivative approach}, Ann. Mat. Pura Appl. (4) \textbf{186} (2007),
  no.~2, 341--358.

\bibitem{FBR2}
Juli{\'a}n Fern{\'a}ndez~Bonder and Julio~D. Rossi, \emph{A nonlinear
  eigenvalue problem with indefinite weights related to the {S}obolev trace
  embedding}, Publ. Mat. \textbf{46} (2002), no.~1, 221--235.

\bibitem{GMSL}
Jorge Garc{\'{\i}}a~Meli{\'a}n and Jos{\'e} Sabina~de Lis, \emph{On the
  perturbation of eigenvalues for the {$p$}-{L}aplacian}, C. R. Acad. Sci.
  Paris S\'er. I Math. \textbf{332} (2001), no.~10, 893--898.

\bibitem{Henrot}
A.~Henrot and M.~Pierre, \emph{Optimization de forme: un analyse g\'eom\'etric.
  mathematics and applications}, vol.~48, Springer-Verlag, 2005.

\bibitem{LSU}
O.~A. Lady{\v{z}}enskaja, V.~A. Solonnikov, and N.~N. Ural{\cprime}ceva,
  \emph{Linear and quasilinear equations of parabolic type}, Translated from
  the Russian by S. Smith. Translations of Mathematical Monographs, Vol. 23,
  American Mathematical Society, Providence, R.I., 1967.

\bibitem{LDT}
Enrique~J. Lami~Dozo and Olaf Torn{\'e}, \emph{Symmetry and symmetry breaking
  for minimizers in the trace inequality}, Commun. Contemp. Math. \textbf{7}
  (2005), no.~6, 727--746.

\bibitem{L}
Gary~M. Lieberman, \emph{Boundary regularity for solutions of degenerate
  elliptic equations}, Nonlinear Anal. \textbf{12} (1988), no.~11, 1203--1219.

\bibitem{MR}
Sandra Mart{\'{\i}}nez and Julio~D. Rossi, \emph{Isolation and simplicity for
  the first eigenvalue of the {$p$}-{L}aplacian with a nonlinear boundary
  condition}, Abstr. Appl. Anal. \textbf{7} (2002), no.~5, 287--293.

\bibitem{S}
M.~W. Steklov, \emph{Sur les probl\'emes fondamentaux en physique
  math\'ematique}, Ann. Sci. Ecole Norm. Sup. \textbf{19} (1902), 445--490.

\bibitem{V}
J.~L. V{\'a}zquez, \emph{A strong maximum principle for some quasilinear
  elliptic equations}, Appl. Math. Optim. \textbf{12} (1984), no.~3, 191--202.

\end{thebibliography}

\end{document}